\documentclass[11pt]{article}
\usepackage[papersize={8.5in,11in}, margin=1.2in]{geometry}
\usepackage{amsmath,amssymb,amsthm,mathtools}

\usepackage{microtype}
\usepackage{subfigure}
\usepackage{booktabs} 
\usepackage{dsfont} 
\usepackage{authblk} 
\usepackage{graphicx} 
\usepackage{algorithmic,algorithm} 
\usepackage[hidelinks]{hyperref}
\usepackage{blindtext}

\usepackage{doi}
\usepackage[numbers,sort]{natbib}
\renewcommand{\bibname}{References}
\renewcommand{\bibsection}{\section*{\bibname}}
\setcitestyle{authoryear,open={(},close={)}} 
\setcitestyle{square,numbers} 
\usepackage[capitalize,noabbrev]{cleveref}

\usepackage{enumitem}
\usepackage{url}

\theoremstyle{plain}
\newtheorem{theorem}{Theorem}[section]
\newtheorem{lemma}[theorem]{Lemma}
\newtheorem{proposition}[theorem]{Proposition}
\newtheorem{corollary}[theorem]{Corollary}

\theoremstyle{definition}

\newtheorem{assumption}{Assumption}
\newtheorem{example}[theorem]{Example}
\theoremstyle{remark}
\newtheorem{remark}[theorem]{Remark}

\newcommand{\R}{\mathbb{R}}
\newcommand{\G}{\mathcal{G}}
\newcommand{\K}{\mathcal{K}}
\newcommand{\norm}[1]{\|#1\|}

\title{Obtaining Game-Stationary Points for Smooth Nonconvex–Nonconcave Minimax Problems via First-Order Methods}
\author[1]{Mengxuan Dong}
\author[2]{Yao Yao}
\author[1]{Jiawei Zhang}
\affil[1]{Department of Computer Sciences, University of Wisconsin–Madison}
\affil[2]{School of Mathematics, University of Minnesota--Twin Cities}
\date{}

\begin{document}

\maketitle

\begin{abstract}
Minimax optimization is a fundamental framework in machine learning, robust optimization, and game theory, yet finding first-order stationary points of general nonconvex--nonconcave minimax problems remains challenging without additional structural assumptions. Existing guarantees often rely on global PL- or KL-type conditions that connect max-player stationarity to global inner optimality, or on Minty-type conditions that impose a global relation on the game gradient field relative to a reference solution; local KL variants relax the former requirement but typically require initialization and tracking within a near-optimal region. Such conditions may be difficult to satisfy in many applications. In contrast, we develop a first-order method that finds an \(\epsilon\)-stationary point within \(\widetilde{\mathcal O}(\epsilon^{-2})\) first-order iterations under a local inverse-Lipschitz regularity condition around approximate max-player stationary points, together with a compactness condition on a penalty sublevel set. Our condition places no optimality requirement on stationary points of the inner maximization problem: they need not be globally, or even locally, maximizing. We further provide sufficient conditions for the required regularity. In the unconstrained setting, it follows from uniform nonsingularity of the maximization-variable Hessian near stationary points; for constrained upper Moreau envelopes, it follows from standard KKT regularity conditions. These results establish first-order complexity guarantees for classes of nonconvex--nonconcave minimax problems not covered by the above PL-, KL-, or Minty-type frameworks.

\end{abstract}

\section{Introduction}
\label{sec:introduction}
Minimax optimization is a basic modeling framework for machine learning and engineering, including
adversarial training, robust optimization, generative modeling, multi-agent systems,
fair representation learning, reinforcement learning and robust
control
\citep{goodfellow2014gan,madry2018adversarial,sinha2018distributional,
madras2018fair,dai2018sbeed,nachum2019dualdice,basar2008hinfinity, ozdaglar2023revisiting}.

A minimax problem takes the following form:
\begin{equation}
\min_{x\in\mathbb{R}^d}\max_{y\in\mathbb{R}^m} F(x,y).
    \label{eq:intro-minimax}
\end{equation}
Convex--concave instances of \eqref{eq:intro-minimax} admit a mature theory
based on monotone variational inequalities
\citep{nemirovski2004prox,mokhtari2020unified,hamedani2021linesearch,
zhang2024robust,zheng2026lastiterate}.
When only the minimizing component is nonconvex, existing approaches include gradient descent–ascent methods, two-timescale algorithms, proximal-point-type methods, and smoothing techniques \citep{lin2020gda,rafique2022weakly,nouiehed2019solving,kong2021accelerated,zhang2020smoothedgda,yang2022faster}.
However, these methods do not readily extend to the \emph{nonconvex–nonconcave} (NC-NC) regime, where the absence of convexity and concavity presents substantial additional challenges. 
Indeed, several existing works \citep{jin2020localminimax,grimmer2023landscape,daskalakis2021complexity, diakonikolas2021efficient} have highlighted the difficulties inherent in this setting. For instance, gradient descent–ascent methods may exhibit cyclic behavior or diverge, the choice of an appropriate local solution concept requires careful consideration and solving the general nonconvex–nonconcave problem is computationally challenging.

Motivated by these challenges, we study the nonconvex-nonconcave setting and our goal is to find a 
\emph{first-order game stationary point}, namely, a
pair $(x,y)$ satisfying
\begin{equation}
    \nabla_x F(x,y)=0,
    \qquad
    \nabla_y F(x,y)=0.
    \label{eq:intro-game-stationarity}
\end{equation}
Accordingly, we consider an $\epsilon$-approximate version of \eqref{eq:intro-game-stationarity} defined by
$\|\nabla_xF(x,y)\|+\|\nabla_yF(x,y)\|\leq\epsilon$. Finding an $\epsilon$-stationary point under nonconvex-nonconcave setting remains a standard computational target \citep{zheng2023universal, li2025primaldual, tsaknakis2021fne}. Although \eqref{eq:intro-game-stationarity} is equivalent to $\nabla F = 0$, applying gradient descent to both variables can easily fail. Standard stationarity guarantees rely on $F$ being bounded below, whereas descending $y$ in minimax problem can drive $F$ to $-\infty$. 

\subsection{Literature review}


Many convergence guarantees for nonconvex--nonconcave problems relate the max-player stationarity residual to the global value $F^\star(x):=\sup_y F(x,y)$. A broad class of results assumes that, on a
feasible set $\mathcal{Y}$,
\(
\operatorname{dist}\bigl(0,-\nabla_yF(x,y)+N_{\mathcal{Y}}(y)\bigr)
\geq c\bigl(F^\star(x)-F(x,y)\bigr)^\theta,
\)
for some $c>0$ and $\theta>0$. This condition includes the unconstrained one-sided Polyak--\L{}ojasiewicz (PL) condition with 
$\theta=\tfrac{1}{2}$ \citep{nouiehed2019solving,yang2022faster}, as well as
the global Kurdyka--\L{}ojasiewicz (KL) condition with more general values of
$\theta$ \citep{zheng2023universal,li2025primaldual}. In particular, these
assumptions imply that every max-player stationary point in their respective
regions of validity is globally optimal.

Recent work relaxes this requirement by restricting the region on which the
inequality is assumed to hold. Specifically,
\citet{luwang2026constrainedminimax} impose a local KL condition on
$x$-dependent near-optimal level sets, thereby allowing suboptimal stationary points outside these regions. Their guarantee, however, requires an initial $y^0$ that is already near-optimal for the inner problem at $x^0$ and uses successive inner solutions to initialize subsequent subproblems. Obtaining the first such solution requires either an initialization within the local KL region or additional favorable structure. Consequently, their guarantee does not cover arbitrary joint initialization.

A different line of work assumes a Minty variational inequality (MVI). In the unconstrained setting, it requires the existence of a point $z^\star$ such that
$\langle G(z),z-z^\star\rangle\geq0$ for all $z$, where
$G=(\nabla_xF,-\nabla_yF)$. This condition implies that $z^\star$ is a global saddle point and, consequently, a game-stationary point, restricting its applicability to general nonconvex--nonconcave objectives, which limits its applicability to broad classes of nonconvex--nonconcave objectives \citep{zheng2023universal}. The weak MVI relaxes this requirement, for example, to $\langle G(z),z-z^\star\rangle\geq-\eta\|G(z)\|^2$ for a game stationary reference point $z^\star$. Although subsequent work has expanded the admissible range of $\eta>0$, the assumption still imposes a global condition on the gradient field relative to a single solution \citep{diakonikolas2021efficient,pethick2022escaping,fan2024weaker,alacaoglu2024revisiting}.

Smooth merit functions provide another approach by reformulating game stationarity as a minimization problem \citep{raghunathan2019gni,tsaknakis2021fne}. Their global minimizers characterize the stationary points of the original game. However, these merit functions are generally nonconvex, so first-order methods typically provide guarantees only for approximate stationarity of the merit objective. Without additional structure, such guarantees do not certify game stationarity of the original problem. The challenge is therefore to construct a tractable reformulation whose approximate stationary points yield the desired certificate.

These limitations of existing approaches motivate the following question:
\begin{center}
\emph{Can we achieve global convergence to game stationarity without requiring max-player stationarity to imply global optimality or imposing global conditions relative to a reference solution?}
\end{center}

We answer this question affirmatively. We develop a penalty method that enforces $y$ to be nearly stationary and prove that approximate stationarity of the penalty function can be transferred to the approximate game stationarity of the original problem under suitable regularity and growth assumptions. The penalty objective is smooth and admits an efficiently computable gradient. Building on these properties, our method computes an $\epsilon$-approximate game stationary point using $\widetilde{O}(\epsilon^{-2})$ total gradient evaluations.

\subsection{Our contributions}

\paragraph{Summary of contributions.}
We address the preceding question through a smooth minimization
reformulation. First, under local regularity and growth assumptions,
we construct a  function whose approximate stationary points
are approximate game stationary points of the original objective,
using a fixed finite penalty parameter. Second, we develop an inexact
gradient descent method with $\widetilde{O}(\varepsilon^{-2})$
first-order oracle complexity. Third, we provide sufficient conditions
for the required regularity, covering nonsingular Hessians at stationary points for max problem
and a class of constrained problems satisfying strict
complementarity and reduced-Hessian nondegeneracy.

\paragraph{Our formulation.}
For an $L$-smooth objective $F$ and fix $p>L$, define the upper proximal point of $F(x,\cdot)$ at $y$ and the associated stationarity gap by
\begin{align}
    w(x,y)
    &:=\arg\max_{v\in\mathbb{R}^m}
    \left\{F(x,v)-\frac{p}{2}\|v-y\|^2\right\},
    \label{eq:w-def}\\
    \mathcal{G}(x,y)
    &:=F(x,w(x,y))
      -\frac{p}{2}\|w(x,y)-y\|^2-F(x,y).
    \label{eq:G-def}
\end{align}
The inner problem in \eqref{eq:w-def} is strongly concave, and its gap $\mathcal{G}$ is nonnegative and vanishes exactly when
$\nabla_yF(x,y)=0$ (Lemma~\ref{lem:gap-properties}). We therefore minimize
\begin{equation}
    P_\rho(x,y):=F(x,y)+\rho\mathcal{G}(x,y).
    \label{eq:Prho-def}
\end{equation}
This penalty function has a Lipschitz gradient that can be evaluated
using first-order information and a strongly concave response solve.

\paragraph{Stationary exactness and complexity.}
Our key result transfers approximate stationarity of $P_\rho$
to stationarity of $F$. Under the stated local regularity,
growth, and compact-sublevel assumptions, a sufficiently large finite $\rho$ ensures
\begin{equation}
    \|\nabla P_\rho(x,y)\|\leq\varepsilon
    \quad\Longrightarrow\quad
    \|\nabla_xF(x,y)\|+\|\nabla_yF(x,y)\|
    \leq C_\rho\varepsilon,
    \label{eq:intro-transfer}
\end{equation}
on the relevant sublevel set whenever the merit gradient is
sufficiently small. Here $C_\rho$ is independent of the target
accuracy. Thus, stationary points of $P_\rho$ are game stationary
points of $F$ (Corollary~\ref{cor:stationary-exactness}).
This transfer yields an inexact gradient descent method requiring
$\widetilde{O}(\varepsilon^{-2})$ gradient evaluations of $F$
to find a $\varepsilon$-game stationary point
(Theorem~\ref{thm:inexact-gd-complexity}). This is the canonical rate of
gradient descent for finding an $\varepsilon$-stationary point of a smooth
nonconvex function, now obtained for game stationarity of a
nonconvex--nonconcave minimax problem.

\paragraph{Local regularity beyond global gradient domination.}
Our local regularity condition compares the proximal displacement
$w(x,y)-y$ with the corresponding change in $\nabla_yF$. This condition is imposed locally near max-player stationary points and holds uniformly over the relevant compact sublevel set
(Assumption~\ref{ass:inverseLip}).
For $C^2$ objectives, Proposition~\ref{prop:hessian-nondegeneracy} shows
that the condition is satisfied whenever $\nabla_{yy}^2F$ is nonsingular at all max-player stationary points on this compact set. In particular, it allows for local minima, maxima and saddle points on the max-player side. (Proposition~\ref{prop:hessian-nondegeneracy}). When strongly convex constraints are introduced, we further show that the corresponding
upper Moreau envelope satisfies this condition under strict complementarity
and nondegeneracy of the reduced Lagrangian Hessian (Proposition~\ref{prop:constrained-envelope-inverseLip}).
\section{Setup and the regularized gap penalty}
\label{sec:regularized-gap-penalty}
Throughout, vector norms denote Euclidean norms, and matrix norms denote operator norms. For the minimax problem \eqref{eq:intro-minimax}, we first make the following assumption.

\begin{assumption}[Smoothness]
\label{ass:smoothness}
The function $F: \mathbb{R}^d \times \mathbb{R}^m \to \mathbb{R}$ is continuously
differentiable, and its gradient is $L$-Lipschitz continuous; that is, for any
$x_1,x_2\in\mathbb{R}^d$ and $y_1,y_2\in\mathbb{R}^m$,
\begin{equation}
    \|\nabla F(x_1, y_1) - \nabla F(x_2, y_2)\| 
    \le L (\|x_1 - x_2\| + \|y_1 - y_2\|).
    \label{eq:smoothness}
\end{equation}
\end{assumption}

Recall that fix $p>L$, the proximal point $w(x,y)$, the stationarity gap $\mathcal{G}(x,y)$, and the penalty function $P_\rho(x,y)$ for $\rho > 0$, are defined in~\eqref{eq:w-def}, \eqref{eq:G-def}, and~\eqref{eq:Prho-def}, respectively. Then we have the following lemma.

\begin{lemma}
\label{lem:gap-properties}
Under Assumption~\ref{ass:smoothness}, the following statements hold.
\begin{enumerate}[label=(\roman*)]
    \item The maximization problem in \eqref{eq:w-def} is
    $(p-L)$-strongly concave and has a unique maximizer $w(x,y)$. Moreover, 
    \begin{equation}
        \nabla_yF(x,w(x,y))=p(w(x,y)-y).
        \label{eq:proximal-foc}
    \end{equation}

    \item The map $w$ is globally Lipschitz with constant $\kappa_w := \frac{\sqrt{L^2 + p^2}}{p-L}$. More precisely, for $z_i = (x_i,y_i)$ and
    $w_i=w(x_i,y_i)$ with $i=1,2$,
    \begin{equation}
        \norm{w_1-w_2}
        \leq
        \frac{L\norm{x_1-x_2}+p\norm{y_1-y_2}}{p-L}
        \leq \frac{\sqrt{L^2 + p^2}}{p-L}\|z_1-z_2\|.
        \label{eq:w-lipschitz}
    \end{equation}

    \item With $w=w(x,y)$, we have
    \begin{equation}
        \G(x,y)
        \geq
        \frac{p-L}{2}\norm{w-y}^2.
        \label{eq:gap-displacement-lower}
    \end{equation}
    In addition,
    \begin{equation}
        \frac{1}{2(p+L)}\norm{\nabla_yF(x,y)}^2
        \leq
        \G(x,y)
        \leq
        \frac{1}{2(p-L)}\norm{\nabla_yF(x,y)}^2.
        \label{eq:gap-gradient-bounds}
    \end{equation}
    Consequently,
    \begin{equation}
        \G(x,y)=0
        \quad\Longleftrightarrow\quad
        \nabla_yF(x,y)=0.
        \label{eq:gap-zero-equivalence}
    \end{equation}

    \item The gap $\G(x,y)$ is continuously differentiable and
    \begin{align}
        \nabla_x\G(x,y)
        &=
        \nabla_xF(x,w)-\nabla_xF(x,y),
        \label{eq:gap-gradient-x}\\
        \nabla_y\G(x,y)
        &=\nabla_yF(x,w)-\nabla_yF(x,y)=p(w-y)-\nabla_yF(x,y).
        \label{eq:gap-gradient-y}
    \end{align}
    Consequently,
    \begin{align}
        \nabla_xP_\rho(x,y)
        &=
        \nabla_xF(x,y)
        +\rho\bigl[\nabla_xF(x,w)-\nabla_xF(x,y)\bigr],
        \label{eq:penalty-gradient-x}\\
        \nabla_yP_\rho(x,y)
        &=
        \nabla_yF(x,y)
        +\rho\bigl[\nabla_yF(x,w)-\nabla_yF(x,y)\bigr].
        \label{eq:penalty-gradient-y}
    \end{align}

    \item The function $P_\rho$ has a globally Lipschitz gradient with constant 
    \begin{equation}
        L_\rho = L + \rho L (1 + \sqrt{1+\kappa_w^2}).
        \label{eq:penalty-smoothness}
    \end{equation}
\end{enumerate}
\end{lemma}

The proof relies primarily on the \((p-L)\)-strong concavity of the problem in \eqref{eq:w-def} and Danskin’s theorem. The detailed proof is provided in Appendix~\ref{app:proof-gap-properties}.

\begin{remark}[Only the max-player curvature constrains $p$]
\label{rem:p-choice}

Let $L_y\leq L$ be a Lipschitz constant for $\nabla_yF(x,\cdot)$,
uniformly in $x$. The condition $p>L$ is used only to ensure the strong
concavity of the inner problem in \eqref{eq:w-def}. Thus, all results remain
valid under the weaker condition $p>L_y$, with every occurrence of $p-L$
replaced by $p-L_y$ in Lemma~\ref{lem:gap-properties}
and Section~\ref{sec:transfer-and-complexity}. 
\end{remark}

Consider $C_0 \ge \max_y F(x^0,y) + 1$ for some initial $x^0$. We make the following assumptions.
\begin{assumption}\label{ass:compact-sublevel}
There exists a constant $\rho_C\geq 1$ and a compact set $\K$ such that
\begin{equation}
    \{
        z:=(x,y)\in\R^{d+m}:P_{\rho}(z)\leq C_0
    \} \subseteq \K, \quad \forall \rho\geq\rho_C.
    \label{eq:common-compact-set}
\end{equation}
\end{assumption}

\begin{assumption}[Local inverse--Lipschitz regularity]
\label{ass:inverseLip}
There exist constants $\eta_s>0$, $\delta_s>0$, and $\kappa_s>0$
such that the following holds.  For every $(x,y)\in\K$, let
$w=w(x,y)$.  If
\begin{equation}
    \max\{
        \norm{\nabla_yF(x,y)},
        \norm{\nabla_yF(x,w)}
    \}
    \leq \eta_s,
    \qquad
    \norm{w-y}\leq\delta_s,
    \label{eq:inverseLip-region}
\end{equation}
then
\begin{equation}
    \norm{w-y}
    \leq
    \kappa_s
    \norm{\nabla_yF(x,w)-\nabla_yF(x,y)}.
    \label{eq:inverseLip-bound}
\end{equation}
\end{assumption}


We next present two settings under which Assumption~\ref{ass:inverseLip} is satisfied. The first setting concerns the nondegeneracy of the max-player Hessian and is formalized in the following proposition.
\begin{proposition}
    \label{prop:hessian-nondegeneracy}
Suppose $F$ is twice continuously differentiable on a neighborhood of $\K$, and there exists $\mu_0>0$ such that
\begin{equation}
    \sigma_{\min}\bigl(\nabla_{yy}^2F(x,y)\bigr)
    \geq \mu_0,
    \label{eq:uniform-nondegeneracy}
\end{equation}
for every $(x,y)\in\K$ satisfying $\nabla_yF(x,y)=0$.
Then Assumption~\ref{ass:inverseLip} holds.
\end{proposition}
Note that \eqref{eq:uniform-nondegeneracy} imposes no sign condition on
$\nabla_{yy}^2F$ and it excludes only zero eigenvalues at max-player critical
points in $\K$. Detailed proof is provided in Appendix~\ref{app:proof-hessian}.

Although problem \eqref{eq:intro-minimax} is unconstrained, practical applications often restrict the max-player to a feasible set. A representative example is budget-constrained adversarial training
\citep{madry2018adversarial}, where inner 
perturbations $u$ must  satisfy $c(u)\leq0$. By leveraging the upper Moreau envelope technique, such problems fit seamlessly into our framework. Moreover, when a perturbation exhausts its budget while the loss continues to rise along the outward normal direction, the constraint prevents further ascent \citep{kanai2023relationship} -- yielding a strictly positive KKT multiplier at an active max-player stationary point. Motivated by this intuition, we give our second sufficient setting for Assumption~\ref{ass:inverseLip}.

\begin{proposition}
\label{prop:constrained-envelope-inverseLip}
Let $\lambda>0$, and suppose that $F$ admits the representation
\begin{equation}
    F(x,y)
    =
    \max_{u:\,c(u)\leq0}
    \left\{
        g(x,u)-\frac{\lambda}{2}\norm{u-y}^2
    \right\}.
    \label{eq:inverseLip-envelope}
\end{equation}
Assume the following conditions.
\begin{enumerate}[label=(\roman*)]
    \item The functions $g$ and $c$ are twice continuously
    differentiable, $\nabla_u g(x,\cdot)$ is $L_g$-Lipschitz uniformly
    in $x$, and $\lambda>L_g$.

    \item The function $c$ is $\mu_c$-strongly convex for some
    $\mu_c>0$ and satisfies the Slater condition: there exists
    $u^{\rm s}$ such that $c(u^{\rm s})<0$.

    \item Define
    \begin{equation*}
        \K^+
        :=
        \left\{
            (x,y):
            \inf_{(\bar x,\bar y)\in\K}
            \norm{(x,y)-(\bar x,\bar y)}
            \leq1
        \right\}.
    \end{equation*}
    There exists $\eta_M>0$ such that strict complementarity holds for
    every KKT point associated with $(x,y)\in\K^+$ satisfying
    $\norm{\nabla_yF(x,y)}\leq\eta_M$.  Namely, if $u=u(x,y)$ is the
    unique maximizer in \eqref{eq:inverseLip-envelope} and
    $\alpha=\alpha(x,y)$ is its multiplier, then
    \begin{equation}
        c(u)=0
        \quad\Longrightarrow\quad
        \alpha>0.
        \label{eq:inverseLip-strict-complementarity}
    \end{equation}

    \item At every KKT point described in part~(iii), define
    \begin{equation*}
        H(x,u,\alpha)
        :=
        -\nabla_{uu}^2g(x,u)+\alpha\nabla^2c(u).
    \end{equation*}
    If $c(u)<0$, then $H(x,u,0)$ is nonsingular.  If $c(u)=0$, let
    $Z(u)$ have orthonormal columns spanning
    $
        \{d:\nabla c(u)^\top d=0\}
    $.
    Then the reduced Lagrangian Hessian
    $
        Z(u)^\top H(x,u,\alpha)Z(u)
    $
    is nonsingular.
\end{enumerate}
Then Assumption~\ref{ass:inverseLip} holds.
\end{proposition}
Note that this result does not require $g(x,\cdot)$ to be convex or concave. The proof is provided in Appendix \ref{app:proof-envelope}.


We next provide a concrete example satisfying Assumption~\ref{ass:inverseLip}.
\begin{example}
Consider a label-adversarial regression with
$g(x,u):=\frac12(x-u)^2+r(x)$,
$r(x):=10x^2+23(\cos x-1)$, and $|u|\leq5$. Its upper-Moreau reformulation is
\begin{equation}
    F(x,y):=
    \max_{u^2 - 25 \le 0}\left\{g(x,u)-(u-y)^2\right\},
    \label{eq:intro-example}
\end{equation}
for which $\max_yF(x,y)=\max_{u^2 - 25 \le 0}g(x,u)$. There is a stationary point $(0,0)$ with $F(0,0)=0$ while $\max_yF(0,y)=25/2$. Hence, any global PL or
gap-based KL inequality imposed at $(0,0)$ fails. Our condition
nonetheless holds, with $\kappa_s=1/2$. Details are shown in Appendix~\ref{ex:robust-regression}.
\end{example}
\section{Stationarity transfer and gradient-descent complexity}\label{sec:transfer-and-complexity}

\subsection{Approximate stationarity transfer}\label{subsec:transfer}

Define
\begin{equation*}
    \underline F_{\K}:=\min_{z\in\K}F(z),
    \qquad
    B:=C_0-\underline F_{\K},
    \qquad
    D:=\sqrt{\frac{2B}{p-L}},
    \qquad
    C_y:=(p+L)D.
\end{equation*}

Note that $B$ is finite and nonnegative. With constants $\eta_s,\delta_s,\kappa_s$ from Assumption~\ref{ass:inverseLip}, we set
\begin{equation}
    \bar\rho
    :=
    \max\left\{
        \rho_c,
        \frac{C_y^2}{\eta_s^2},
        \frac{p^2D^2}{\eta_s^2},
        \frac{D^2}{\delta_s^2},
        2(1+p\kappa_s)
    \right\}.
    \label{eq:inverseLip-threshold}
\end{equation}

We are now ready to establish the relationship between the stationary points of \(P_{\rho}\) and \(F\) in the following theorem and corollary.
\begin{theorem}
\label{thm:stationarity-transfer}
Suppose Assumptions~\ref{ass:smoothness}--\ref{ass:inverseLip} hold.  Let
$\rho\geq\bar\rho$, and suppose that $(x,y)$ satisfies
\begin{equation}
    P_\rho(x,y)\leq C_0,
    \qquad
    \norm{\nabla P_\rho(x,y)}\leq\varepsilon.
    \label{eq:inverseLip-low-level}
\end{equation}
Let $w=w(x,y)$ and let $\Delta_y= \nabla_y F(x,w) - \nabla_y F(x,y)$ be the gradient difference. Then
\begin{align}
    \norm{\Delta_y}
    \leq
    \frac{\varepsilon}{\rho-1-p\kappa_s},
    \qquad
    &\norm{w-y}
    \leq
    \frac{\kappa_s\varepsilon}{\rho-1-p\kappa_s},
    \label{eq:inverseLip-delta-displacement}\\
    \norm{\nabla_yF(x,y)}
    \leq
    \frac{(1+p\kappa_s)\varepsilon}
    {\rho-1-p\kappa_s},
    \qquad
    &\norm{\nabla_xF(x,y)}
    \leq
    \left(
        1+
        \frac{\rho L\kappa_s}{\rho-1-p\kappa_s}
    \right)\varepsilon.
    \label{eq:inverseLip-nabla}
\end{align}
Consequently,
\begin{equation}
    \norm{\nabla_xF(x,y)}+\norm{\nabla_yF(x,y)}
    \leq C_\rho\varepsilon,
    \qquad
    C_\rho
    :=
    1+
    \frac{\rho L\kappa_s+1+p\kappa_s}
    {\rho-1-p\kappa_s}.
    \label{eq:inverseLip-game}
\end{equation}
In particular, every stationary point of $P_\rho$ in the common low-level
region is a stationary point of $F$.
\end{theorem}
The proof is provided in Appendix~\ref{app:transfer}

\begin{corollary}
\label{cor:stationary-exactness}
Under the assumptions of Theorem~\ref{thm:stationarity-transfer}, let
$\rho\geq\bar\rho$. If
\begin{equation}
    P_\rho(x,y)\leq C_0,
    \qquad
    \nabla P_\rho(x,y)=0,
\end{equation}
then
\begin{equation}
    \nabla_xF(x,y)=0,
    \qquad
    \nabla_yF(x,y)=0.
\end{equation}
\end{corollary}

\begin{proof}
Apply Theorem~\ref{thm:stationarity-transfer} with $\varepsilon=0$.
Then \eqref{eq:inverseLip-delta-displacement} gives $w(x,y)=y$, and
\eqref{eq:inverseLip-game} gives the conclusion.
\end{proof}

\subsection{Exact gradient descent on the penalty objective}\label{exact-gd}
We first show how to obtain an initial point $z^0=(x^0,y^0)$ in the common low-level set, that is, $P_\rho(z^0)\leq C_0$ for a prescribed $\rho\geq\bar\rho$. 
\begin{proposition}
\label{lem:initial-point}
Suppose Assumption~\ref{ass:smoothness} holds and fix $\rho>0$. Starting from any $u^0\in\R^m$, consider the gradient ascent iteration $ u^{t+1}=u^t+\alpha \nabla_yF(x^0,u^t)$ within $T:=\left\lceil
        \frac{\rho\,\bigl(\overline F(x^0)-F(x^0,u^0)\bigr)}{\alpha(2-\alpha L)(p-L)}
    \right\rceil$ steps.
Then there exists $0\leq t\leq T$ such that
 $\norm{\nabla_yF(x^0,u^t)}^2\leq\frac{2(p-L)}{\rho}$,
    and consequently    
    $P_\rho(x^0,u^t)\leq C_0$.
In particular, $y^0:=u^t$ yields $P_\rho(x^0,y^0)\leq C_0$.
\end{proposition}
\begin{proof}
Since $F(x^0,\cdot)$ is $L$-smooth, bounded above by $\overline F(x^0)$, and $\alpha < \frac{1}{L}$, the standard convergence analysis of gradient ascent with step size $\alpha \le \frac{1}{L}$ gives
$
    \min_{0\leq t\leq T}\norm{\nabla_yF(x^0,u^t)}^2
    \leq
    \frac{2(p-L)}{\rho}
$.
 For such $t$, the upper bound in \eqref{eq:gap-gradient-bounds} yields
\begin{equation*}
    P_\rho(x^0,u^t)
    =
    F(x^0,u^t)+\rho\,\G(x^0,u^t)
    \leq
    \overline F(x^0)+\frac{\rho}{2(p-L)}\norm{\nabla_yF(x^0,u^t)}^2
    \leq
    \overline F(x^0)+1
    \leq
    C_0.
    \qedhere
\end{equation*}
\end{proof}
Since $\rho$ is fixed throughout, $T$ is a constant independent of the target accuracy $\epsilon$, and this initialization does not affect the complexity bounds below.
For a fixed $\rho\geq\bar\rho$, let $z^0:=(x^0, y^0)$ be an
initial point satisfying Assumption~\ref{ass:compact-sublevel}, and consider
\begin{equation}
    z^{k+1}
    =z^k-\frac{1}{L_\rho}\nabla P_\rho(z^k).
    \label{eq:penalty-gradient-descent}
\end{equation}
Here $L_\rho$ is given in \eqref{eq:penalty-smoothness}. Equivalently,
if $w^k=w(x^k,y^k)$, the gradient in
\eqref{eq:penalty-gradient-descent} is computed from
\begin{equation*}
    \nabla P_\rho(x^k,y^k)
    =
    \nabla F(x^k,y^k)
    +\rho\bigl[
        \nabla F(x^k,w^k)-\nabla F(x^k,y^k)
    \bigr].
\end{equation*}
Given $w^k$, the outer gradient requires evaluations of $\nabla F$ at
$(x^k,y^k)$ and $(x^k,w^k)$. The point $w^k$ is the unique solution of
the strongly concave problem in \eqref{eq:w-def}. We next establish the convergence complexity of the exact gradient descent method.

\begin{theorem}
\label{thm:gd-complexity}
Suppose Assumptions~\ref{ass:smoothness}--\ref{ass:inverseLip} hold,
and fix $\rho\geq\bar\rho$. Let
\begin{equation}
    P_c^*:=\min_{z\in\K}P_{\rho_c}(z),
    \qquad
    \Delta_c:=C_0-P_c^*.
    \label{eq:common-descent-gap}
\end{equation}
Then the iterates in \eqref{eq:penalty-gradient-descent} satisfy
\begin{equation}
    P_\rho(z^{k+1})
    \leq
    P_\rho(z^k)
    -\frac{1}{2L_\rho}\norm{\nabla P_\rho(z^k)}^2,
    \label{eq:penalty-descent}
\end{equation}
and remain in the common compact set $\K$. Moreover, for every integer
$N\geq1$,
\begin{equation}
    \min_{0\leq k<N}
    \norm{\nabla P_\rho(z^k)}^2
    \leq
    \frac{2L_\rho\Delta_c}{N}.
    \label{eq:penalty-complexity-bound}
\end{equation}
Consequently, if
\begin{equation}
    N
    \geq
    \frac{2L_\rho\Delta_c C_\rho^2}{\epsilon^2},
    \label{eq:iteration-complexity}
\end{equation}
then at least one iterate $z^k=(x^k,y^k)$ satisfies
\begin{equation}
    \norm{\nabla_xF(x^k,y^k)}
    +\norm{\nabla_yF(x^k,y^k)}
    \leq\epsilon.
    \label{eq:epsilon-game-stationarity}
\end{equation}
For every fixed $\rho\geq\bar\rho$, the outer iteration complexity is
$O(\epsilon^{-2})$.
\end{theorem}

The proof in Appendix~\ref{app:proof-gd} applies the descent lemma to \(P_\rho\), uses \(P_\rho\geq P_{\rho_c}\) to ensure that all iterates remain in \(\mathcal K\), and invokes Theorem~\ref{thm:stationarity-transfer} at the iterate with the smallest gradient norm.

\subsection{Inexact gradient descent on the penalty objective}\label{subsec:inexact-gd}

In practice, exact optimal solutions are often not available, and we instead obtain an approximate
best response \(\widetilde w^k\) satisfying
\begin{equation}
\|\widetilde w^k-w(x^k,y^k)\|
\leq \delta ,
\label{eq:inexact-best-response}
\end{equation}
for some \(\delta>0\). We then replace the exact penalty gradient by an
inexact gradient \(\widetilde g^k\) and perform the update
\begin{equation}
z^{k+1}
=
z^k-\frac{1}{L_\rho}\widetilde g^k, \qquad \widetilde g^k
:=
\nabla F(x^k,y^k)
+\rho\bigl[
\nabla F(x^k,\widetilde w^k)-\nabla F(x^k,y^k)
\bigr].
\label{eq:inexact-penalty-gradient-descent}
\end{equation}

Given \eqref{eq:inexact-best-response}, the inexact gradient error satisfies
\begin{equation}
\|\widetilde g^k-\nabla P_\rho(z^k)\|
= \rho \|\nabla F(x^k,\widetilde w^k)-\nabla F(x^k,w^k)\|
\leq \rho L \|\widetilde w^k-w^k\|.
\label{eq:inexact-gradient-error}
\end{equation}
We next establish the convergence complexity of the inexact gradient descent method.

\begin{theorem}
\label{thm:inexact-gd-complexity}
Suppose Assumptions~\ref{ass:smoothness}--\ref{ass:inverseLip} hold,
and fix $\rho\geq\bar\rho$. Let $P_c^*$ and $\Delta_c$ be defined as
in \eqref{eq:common-descent-gap}, and suppose that
\eqref{eq:inexact-best-response} holds at every iteration. Then the iterates in \eqref{eq:inexact-penalty-gradient-descent} satisfy
\begin{equation}
P_\rho(z^{k+1})
\leq
P_\rho(z^k)
-\frac{1}{2L_\rho}
\|\nabla P_\rho(z^k)\|^2
+\frac{\rho^2L^2\delta^2}{2L_\rho}.
\label{eq:inexact-penalty-descent}
\end{equation}

Moreover, for every integer $N\geq1$ there exists an index
$k$ with $0\leq k<N$ such that
\begin{equation}
P_\rho(z^j)\leq C_0
\quad\text{for all }0\leq j\leq k,
\qquad\text{in particular }
P_\rho(z^k)\leq C_0
\text{ and } z^k\in\K,
\label{eq:inexact-iterate-in-level-set}
\end{equation}
and
\begin{equation}
\|\nabla P_\rho(z^k)\|^2
\leq
\frac{2L_\rho\Delta_c}{N}
+\rho^2L^2\delta^2.
\label{eq:inexact-penalty-complexity}
\end{equation}
For this iterate,
\begin{equation}
\|\nabla_xF(x^k,y^k)\|
+
\|\nabla_yF(x^k,y^k)\|
\leq
C_\rho
\sqrt{
\frac{2L_\rho\Delta_c}{N}
+\rho^2L^2\delta^2
}.
\label{eq:inexact-game-stationarity}
\end{equation}

In particular, if \eqref{eq:iteration-complexity} holds and $\delta\leq\epsilon$,
then at least one iterate $z^k=(x^k,y^k)$ satisfies
\begin{equation}
\|\nabla_xF(x^k,y^k)\|
+
\|\nabla_yF(x^k,y^k)\|
\leq
(1+C_\rho \rho L)\epsilon.
\end{equation}
For every fixed \(\rho\geq\bar\rho\), choosing
\(\delta=O(\epsilon)\) therefore preserves the
\(O(\epsilon^{-2})\) outer iteration complexity.
\end{theorem}

The proof, in Appendix~\ref{app:proof-inexact}, follows that of Theorem~\ref{thm:gd-complexity}: the gradient error \eqref{eq:inexact-gradient-error} contributes the additive term, and descent holds while $\norm{\nabla P_\rho(z^k)}\geq\rho L\delta$, which keeps the selected iterate in $\K$.

\begin{corollary}[Total gradient complexity]
The best-response problem in \eqref{eq:w-def} is strongly concave
and smooth, so gradient ascent initialized at $y^k$ returns
$\widetilde w^k$ satisfying \eqref{eq:inexact-best-response} after
$\mathcal{O}(\log(1/\delta))$ gradient evaluations.
Taking $\delta\le \frac{\epsilon}{\sqrt{2}C_\rho \rho L}$,
the total number of gradient evaluations of $F$ needed to reach an
$\epsilon$-stationary point is
$\widetilde{\mathcal O}(\epsilon^{-2})$.    
\end{corollary}

\section{Numerical experiments}
\label{sec:nr-experiments}
In this section we evaluate our method through quadratic examples and compare its performance with other relevant algorithms.
We consider the quadratic family
\begin{equation}
\label{eq:nr-quadratic}
  \min_{x\in\mathbb R^d}\;\max_{y\in\mathbb R^{2d}}
  F_d(x,y)
  :=-\frac12\|x\|_2^2
  +x^\top\begin{bmatrix}I_d&2I_d\end{bmatrix}Q_d^\top y
  +\frac12 y^\top Q_d
  \begin{bmatrix}I_d&0\\0&-I_d\end{bmatrix}Q_d^\top y,
\end{equation}
where $Q_d\in\mathbb R^{2d\times2d}$ is orthogonal. We test $d=1$
with $Q_1=I_2$ and $d=10$ with a fixed random orthogonal $Q_{10}$,
giving problems of total dimension $3$ and $30$, respectively.
In the coordinates $\widetilde y=Q_d^\top y$, the objective decomposes
into $d$ identical three-variable quadratic blocks. Consequently, both
problems have the unique game stationary point $z^\star=(0,0)$, while
$\nabla^2_{yy}F_d$ has eigenvalues $+1$ and $-1$, each with multiplicity
$d$. Thus $y^\star$ is a saddle point of $F_d(x^\star,\cdot)$.
The rotation mixes the displayed $y$-coordinates but preserves the
Euclidean geometry and the dynamics of the gradient methods under the
corresponding change of variables.

In both experiments, we compare inexact gradient descent on $P_\rho$
with joint GD on $F_d$ using step size $0.05$ and simultaneous GDA with
$(\tau_x,\tau_y)\in\{(0.05,0.05),(0.01,0.2),(0.1,0.01)\}$.
All methods use the same initial point within each experiment.
Each regularized response is computed by gradient ascent initialized
at the current $y^k$. We report the distance $\|z^k-z^\star\|_2$ and
the stationarity residual
\[
 r_k:=\|\nabla_xF_d(z^k)\|_2+\|\nabla_yF_d(z^k)\|_2,
\]
both against outer updates and, for the residual, against gradient
evaluations including all inner iterations. Runs are terminated upon
reaching $r_k\le10^{-6}$ or $\|z^k\|_2\ge10^6$.

For the three-dimensional experiment, we take $z^0=(1,-1,2)$ and
$\rho=16$, with outer step size $1/22$ and inner ascent step size $1/6$.
For the thirty-dimensional experiment, we use a random unit-norm $z^0$,
$p=3.75$, and $\rho=9.50$, with outer step size $0.076$ and inner
ascent step size $0.21$. In the latter experiment, the inner solve stops
when the norm of the gradient of the regularized inner objective is at
most $10^{-10}$.

\begin{figure}[!tbp]
  \centering
  \includegraphics[width=\linewidth]{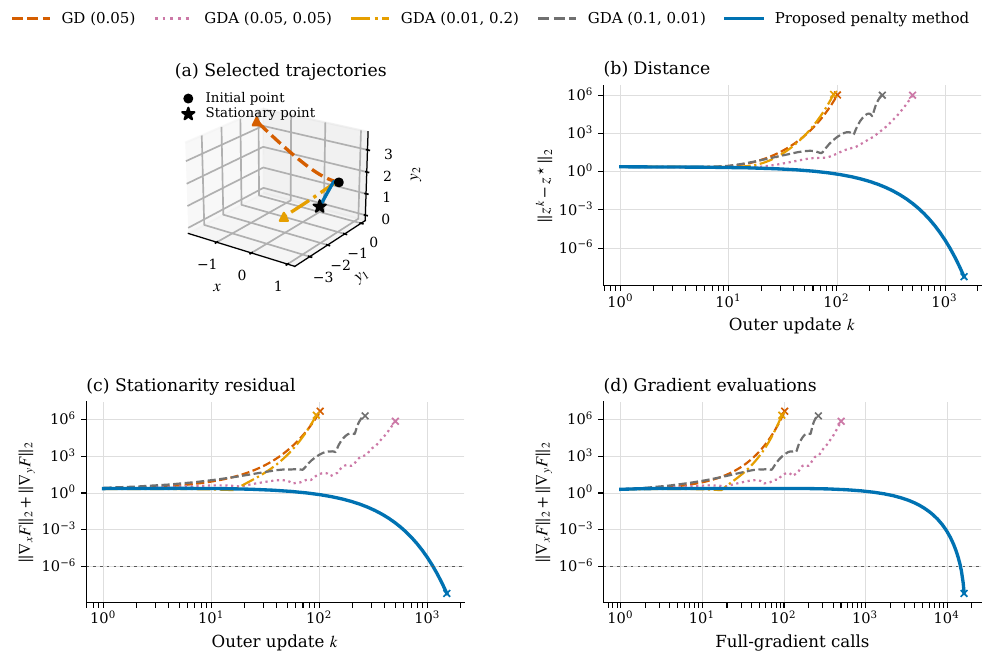}
  \caption{Three-dimensional comparison for \eqref{eq:nr-quadratic}
  with $d=1$, $Q_1=I_2$, and $z^0=(1,-1,2)$.
  (a) Trajectories up to their first exit from $\|z\|_2\le4$.
  (b) Distance to $z^\star$.
  (c) Stationarity residual versus outer update.
  (d) Residual versus gradient evaluations, including inner solves.}
  \label{fig:nr-quadratic}
\end{figure}

\begin{figure}[!tbp]
  \centering
  \includegraphics[width=\linewidth]{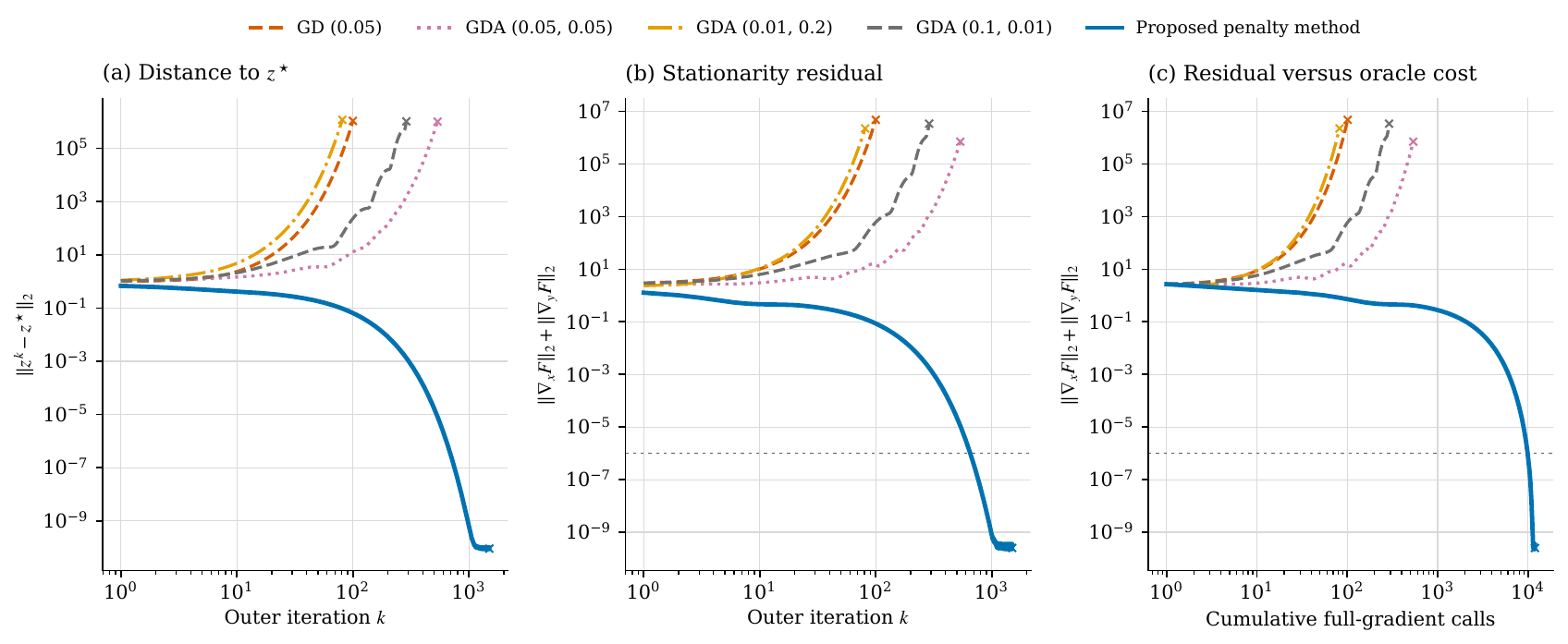}
  \caption{Thirty-dimensional comparison for \eqref{eq:nr-quadratic}
  with $d=10$, a fixed random orthogonal $Q_{10}$, and a common random
  unit-norm initialization.
  (a) Distance to $z^\star$.
  (b) Stationarity residual versus outer update.
  (c) Residual versus gradient evaluations, including inner solves.}
  \label{fig:nr-quadratic-30d}
\end{figure}

The proposed method reaches $r_k\le10^{-6}$ in both experiments:
$1{,}120$ outer updates and $14{,}586$ gradient evaluations in dimension
$3$, and $654$ outer updates and $9{,}910$ gradient evaluations in
dimension $30$. In contrast, GD and all three GDA configurations reach
$\|z^k\|_2\ge10^6$ within $93$--$501$ updates in dimension $3$ and
$81$--$536$ updates in dimension $30$
(Figures~\ref{fig:nr-quadratic} and~\ref{fig:nr-quadratic-30d}).
These results illustrate convergence of the proposed method to game
stationarity despite indefinite max-player curvature. Since the rotated
problem preserves the block spectrum and uses different numerical
parameters, the iteration counts are not intended as a comparison of
convergence speed across dimensions.

\subsubsection*{AI Disclosure Statement}
GPT-5.6 was utilized solely for editing and improving the clarity and readability of the writing in this paper.

\newpage
\bibliography{regularized_gap_penalty_introduction1_revised}

\newpage
\appendix

\section{Proofs}
\label{app:proofs}

\subsection{Proof of Lemma~\ref{lem:gap-properties}}
\label{app:proof-gap-properties}

\begin{proof}
\begin{enumerate}[label=(\roman*)]
\item For fixed $(x,y)$, write $w=w(x,y)$ and define
\begin{equation*}
    q_{x,y}(v)
    :=F(x,v)-\frac{p}{2}\norm{v-y}^2.
\end{equation*}
Since $F(x,\cdot)$ is $L$-smooth, it is $L$-weakly concave.
Hence $q_{x,y}$ is $(p-L)$-strongly concave, proving uniqueness of
$w(x,y)$. Moreover, the first-order condition of $q_{x,y}$ is exactly \eqref{eq:proximal-foc}.

\item Write the first-order conditions \eqref{eq:proximal-foc} at $z_1$ and
$z_2$:
\begin{equation}
    p(w_i-y_i)=\nabla_yF(x_i,w_i),
    \qquad i=1,2.
\end{equation}
After subtraction, rearranging the terms and applying the smoothness of $F$ in Assumption~\ref{ass:smoothness},
\begin{align*}
    p\norm{w_1-w_2}
    &\leq
    p\norm{y_1-y_2}
    +\norm{\nabla_yF(x_1,w_1)-\nabla_yF(x_2,w_2)}\\
    &\leq
    p\norm{y_1-y_2}
    +L\norm{x_1-x_2}
    +L\norm{w_1-w_2}.
\end{align*}
Rearranging gives
\begin{align*}
    \|w_1-w_2\| \leq \frac{L\norm{x_1-x_2}+p\norm{y_1-y_2}}{p-L} \leq \frac{\sqrt{L^2 + p^2}}{p-L}\|z_1-z_2\|,
\end{align*}
where the second inequality follows from the Cauchy--Schwarz inequality.

\item Strong concavity and $\nabla q_{x,y}(w)=0$ imply
\begin{equation*}
    q_{x,y}(w)-q_{x,y}(y)
    \geq
    \frac{p-L}{2}\norm{w-y}^2.
\end{equation*}
Note that the left-hand side is $\G(x,y)$ since $q_{x,y}(y) = F(x,y)$, which proves \eqref{eq:gap-displacement-lower}.

To prove \eqref{eq:gap-gradient-bounds}, let $g:=\nabla_yF(x,y)$ and
substitute $v=y+s$ in \eqref{eq:w-def}--\eqref{eq:G-def}, so that
\begin{equation}
    \G(x,y)=\max_{s\in\R^m}\phi(s),
    \qquad
    \phi(s):=F(x,y+s)-F(x,y)-\frac{p}{2}\norm{s}^2 .
    \label{eq:gap-as-max-in-s}
\end{equation}
Applying the smoothness of $F(x,\cdot)$ and subtracting $\frac{p}{2}\norm{s}^2$ from each side, for all $s$, we have
\begin{equation}
    \langle g,s\rangle-\frac{p+L}{2}\norm{s}^2
    \;\leq\;
    \phi(s)
    \;\leq\;
    \langle g,s\rangle-\frac{p-L}{2}\norm{s}^2 .
    \label{eq:phi-two-sided}
\end{equation}
For any $b>0$ one has
$\max_{s}\{\langle g,s\rangle-\frac{b}{2}\norm{s}^2\}=\norm{g}^2/(2b)$
attained at $s=g/b$. Taking the maximum over $s$ in the right-hand inequality of \eqref{eq:phi-two-sided} and using $b=p-L>0$ proves the upper bound in \eqref{eq:gap-gradient-bounds}.

For the lower bound, note that the maximum in
\eqref{eq:gap-as-max-in-s} dominates the value of $\phi$ at any single point.
Evaluating the left-hand inequality of \eqref{eq:phi-two-sided} at the choice
$s=g/(p+L)$, which maximizes that lower bound, gives
\begin{equation*}
    \G(x,y)
    \geq
    \phi\Bigl(\frac{g}{p+L}\Bigr)
    \geq
    \frac{\norm{g}^2}{2(p+L)} .
\end{equation*}
The equivalence \eqref{eq:gap-zero-equivalence} follows immediately.

\item Because the maximizer is unique, Danskin's theorem gives
\eqref{eq:gap-gradient-x} and the first equality in
\eqref{eq:gap-gradient-y}. The second equality in \eqref{eq:gap-gradient-y} follows from
\eqref{eq:proximal-foc}. Equations \eqref{eq:penalty-gradient-x} and
\eqref{eq:penalty-gradient-y} follow from
$P_\rho=F+\rho\G$.

\item Define
$\widehat z:=(x,w(x,y))$. Since the norm on $(x,y)$ is Euclidean, $\norm{x_1-x_2}\leq\norm{z_1-z_2}$, and
the bound $\norm{w_1-w_2}\leq\kappa_w\norm{z_1-z_2}$ from \eqref{eq:w-lipschitz} yields
\begin{equation*}
    \norm{\widehat z_1-\widehat z_2}^2
    = \norm{x_1-x_2}^2+\norm{w_1-w_2}^2
    \leq (1+\kappa_w^2)\norm{z_1-z_2}^2 .
\end{equation*}
Hence, the map $z\mapsto\widehat z$ is $\sqrt{1+\kappa_w^2}$-Lipschitz. Equations
\eqref{eq:gap-gradient-x}--\eqref{eq:gap-gradient-y} imply
  $\nabla\G(z)=\nabla F(\widehat z)-\nabla F(z)$. Hence, by the triangle inequality and the Lipschitz continuity of $\nabla F$, we have
\begin{align*}
\norm{\nabla\G(z_1)-\nabla\G(z_2)} &\leq \|\nabla F(\widehat z_1)-\nabla F(\widehat z_2)\| + \|\nabla F(z_1)-\nabla F(z_2)\| \\
&\le L \|\widehat z_1 - \widehat z_2\| + L \|z_1 - z_2\| \le L (1 + \sqrt{1 + \kappa_w^2}) \|z_1 - z_2\|.
\end{align*}
Together with $P_\rho = F + \rho\G$, we have
\begin{align*}
    \norm{\nabla P_\rho(z_1)-\nabla P_\rho(z_2)}
    &\leq \|\nabla F(z_1)-\nabla F(z_2)\| + \rho \|\nabla \G ( z_1)-\nabla \G(z_2)\|  \\
    &\leq L\|z_1 - z_2\| + \rho L (1 + \sqrt{1 + \kappa_w^2}) \|z_1 - z_2\| \\
    &= (L + \rho L (1 + \sqrt{1+\kappa_w^2})) \|z_1 - z_2\|,
\end{align*}
where we are applying the Lipschitz continuity of $\nabla F$ and $\nabla \G$ in the second inequality. This proves that $P_\rho$ has a globally Lipschitz gradient with constant $L_\rho = L + \rho L (1 + \sqrt{1+\kappa_w^2})$.
\end{enumerate}
\end{proof}



\subsection{Proofs of the sufficient conditions for Assumption~\ref{ass:inverseLip}}
\label{app:secant}

\subsubsection{Proof of Proposition~\ref{prop:hessian-nondegeneracy}}
\label{app:proof-hessian}
\begin{proof}
We first claim that there exists $\eta_0>0$ such that
\begin{equation}
    \sigma_{\min}\bigl(\nabla_{yy}^2F(x,y)\bigr)\geq\frac{3\mu_0}{4}
    \text{ whenever }
    (x,y)\in\K
    \text{ and }
    \norm{\nabla_yF(x,y)}\leq\eta_0 .
    \label{eq:hessian-near-critical-bound}
\end{equation}
Suppose not, then there would exist a sequence
$z_j=(x_j,y_j)\in\K$ such that
\begin{equation*}
    \norm{\nabla_yF(x_j,y_j)}\to 0,
    \qquad
    \sigma_{\min}\bigl(\nabla_{yy}^2F(x_j,y_j)\bigr)
    <\frac{3\mu_0}{4}.
\end{equation*}
Compactness of $\K$ gives a convergent subsequence with limit
$z_\star=(x_\star,y_\star)\in\K$. By continuity,
\begin{equation*}
    \nabla_yF(x_\star,y_\star)=0,
    \qquad
    \sigma_{\min}\bigl(\nabla_{yy}^2F(x_\star,y_\star)\bigr)
    \leq\frac{3\mu_0}{4},
\end{equation*}
which contradicts \eqref{eq:uniform-nondegeneracy}. Let $U$ be an open
neighborhood of $\K$ on which $F$ is twice continuously differentiable.
Since $\K$ is compact, there exists $r_1>0$ such that the closed enlargement
$
    \K_{r_1}
    :=
    \{
        z'\in\R^{d+m}:
        \inf_{z\in\K}\norm{z'-z}\leq r_1
    \}$
is contained in $U$. As $\K_{r_1}$ is compact,
$\nabla_{yy}^2F$ is uniformly continuous on $\K_{r_1}$. Therefore, there
exists $r_0\in(0,r_1]$ such that
\begin{equation}
    \norm{\nabla_{yy}^2F(x,v)-\nabla_{yy}^2F(x,y)}
    \leq \frac{\mu_0}{4}, \text{ for all } (x,y)\in\K \text{ with } \norm{v-y}\leq r_0.
    \label{eq:hessian-uniform-continuity}
\end{equation}

Set $\eta_s:=\min\{\eta_0,\,pr_0\}$ and suppose that
$(x,y)\in\K$, $w=w(x,y)$, and
$\max\{\norm{\nabla_yF(x,y)},\norm{\nabla_yF(x,w)}\}\leq\eta_s$.
The first-order condition \eqref{eq:proximal-foc} then gives
$\norm{w-y}=\norm{\nabla_yF(x,w)}/p\leq r_0$.  Define
\begin{equation*}
    A(x,y)
    :=
    \int_0^1
    \nabla_{yy}^2F(x,y+t(w-y))\,dt.
\end{equation*}
Then
$
    \norm{A(x,y)-\nabla_{yy}^2F(x,y)}
    \leq
    \int_0^1
    \norm{
        \nabla_{yy}^2F\bigl(x,y+t(w-y)\bigr)
        -\nabla_{yy}^2F(x,y)
    }\,dt \leq \frac{\mu_0}{4}$.
Combining this estimate with
\eqref{eq:hessian-near-critical-bound} and the singular-value
perturbation inequality proves
\begin{align*}
\sigma_{\min}\bigl(A(x,y)\bigr)
    \geq
    \sigma_{\min}\bigl(\nabla_{yy}^2F(x,y)\bigr)
    -\norm{A(x,y)-\nabla_{yy}^2F(x,y)}
    \geq \frac{3\mu_0}{4}-\frac{\mu_0}{4}
    =\frac{\mu_0}{2}.
\end{align*}

The fundamental theorem of calculus gives
\begin{equation*}
    \nabla_yF(x,w)-\nabla_yF(x,y)
    =A(x,y)(w-y).
\end{equation*}
Consequently,
\begin{equation*}
    \norm{w-y}
    \leq
    \frac{2}{\mu_0}
    \norm{\nabla_yF(x,w)-\nabla_yF(x,y)}.
\end{equation*}
Thus Assumption~\ref{ass:inverseLip} holds with
$\eta_s=\min\{\eta_0, \, pr_0\}$, $\delta_s=r_0$, and
$\kappa_s=2/\mu_0$.
We note that the proximal structure of $w$ entered only through
\eqref{eq:proximal-foc}, which was used to bound $\norm{w-y}$; the
argument shows more generally that
$\norm{v-y}\leq\frac{2}{\mu_0}\norm{\nabla_yF(x,v)-\nabla_yF(x,y)}$
for every $(x,y)\in\K$ with $\norm{\nabla_yF(x,y)}\leq\eta_0$ and every
$v$ with $\norm{v-y}\leq r_0$.
\end{proof}

\subsubsection{Proof of Proposition~\ref{prop:constrained-envelope-inverseLip}}
\label{app:proof-envelope}
\begin{proof}
Throughout the proof we write
\begin{equation*}
    \phi_{x,y}(u):=g(x,u)-\frac{\lambda}{2}\norm{u-y}^2,
    \qquad
    \mathcal C:=\{u:c(u)\leq0\},
\end{equation*}
so that $F(x,y)=\max_{u\in\mathcal C}\phi_{x,y}(u)$. The proof
proceeds in five steps.

\paragraph{Step 1: Well-posedness, continuity, and the residual map.}
Since $\nabla_ug(x,\cdot)$ is $L_g$-Lipschitz and $\lambda > L_g$, $\phi_{x,y}$ is $(\lambda-L_g)$-strongly
concave. The set $\mathcal C$ is closed and
convex because $c$ is continuous and convex, nonempty by the Slater
condition, and bounded because the strongly convex function $c$ is
coercive; thus $\mathcal C$ is compact. Consequently the maximizer
$u(x,y)$ of $\phi_{x,y}$ over $\mathcal C$ exists and is unique, and it
is characterized by the variational inequality
\begin{equation}
    \langle\nabla\phi_{x,y}(u(x,y)),\,v-u(x,y)\rangle\leq0
    \qquad\text{for all }v\in\mathcal C.
    \label{eq:inverseLip-VI}
\end{equation}
Fix two parameter pairs $(x,y)$ and $(x',y')$ and abbreviate
$u=u(x,y)$, $u'=u(x',y')$. Applying \eqref{eq:inverseLip-VI} at $(x,y)$
with $v=u'$ and at $(x',y')$ with $v=u$, and adding the two
inequalities, gives
$\langle\nabla\phi_{x,y}(u)-\nabla\phi_{x',y'}(u'),u-u'\rangle\geq0$.
Inserting $\pm\nabla\phi_{x,y}(u')$ and using strong concavity of
$\phi_{x,y}$ yields
\begin{equation*}
    (\lambda-L_g)\norm{u-u'}^2
    \leq
    -\langle\nabla\phi_{x,y}(u)-\nabla\phi_{x,y}(u'),u-u'\rangle
    \leq
    \langle\nabla\phi_{x,y}(u')-\nabla\phi_{x',y'}(u'),u-u'\rangle .
\end{equation*}
Since
$\nabla\phi_{x,y}(u')-\nabla\phi_{x',y'}(u')
=\nabla_ug(x,u')-\nabla_ug(x',u')+\lambda(y-y')$, applying the Cauchy--Schwarz inequality to the right-hand side, we conclude that
\begin{equation}
    \norm{u(x,y)-u(x',y')}
    \leq
    \frac{1}{\lambda-L_g}
    \Bigl(
        \norm{\nabla_ug(x,u(x',y'))-\nabla_ug(x',u(x',y'))}
        +\lambda\norm{y-y'}
    \Bigr).
    \label{eq:inverseLip-u-continuity}
\end{equation}
Since $\nabla_ug$ is continuous, \eqref{eq:inverseLip-u-continuity}
shows that $(x,y)\mapsto u(x,y)$ is continuous on $\R^d\times\R^m$, and
with $x=x'$ it gives the Lipschitz bound
\begin{equation}
    \norm{u(x,y)-u(x,y')}
    \leq
    \frac{\lambda}{\lambda-L_g}\norm{y-y'}.
    \label{eq:inverseLip-u-lipschitz}
\end{equation}
Since $\mathcal C$ is compact, $\phi_{x,y}$ has a unique maximizer, and
$\nabla_y\phi_{x,y}(u)=\lambda(u-y)$ is continuous, Danskin's theorem
shows that $F(x,\cdot)$ is continuously differentiable with
\begin{equation}\label{eq:def-Rx}
    R_x(y)
    :=
    \nabla_yF(x,y)
    =
    \lambda(u(x,y)-y).
\end{equation}
In particular $R:(x,y)\mapsto R_x(y)$ is continuous, and by
\eqref{eq:inverseLip-u-lipschitz},
\begin{equation*}
    \norm{R_x(y)-R_x(y')}
    \leq
    \lambda\bigl(\norm{u(x,y)-u(x,y')}+\norm{y-y'}\bigr)
    \leq
    L_F\norm{y-y'},
    \qquad
    L_F:=\frac{\lambda(2\lambda-L_g)}{\lambda-L_g}.
\end{equation*}
Hence $\nabla_y F(x,\cdot)$ is $L_F$-Lipschitz continuous. This is the
only property needed in the proof of Lemma~\ref{lem:gap-properties}(i):
for any $p>L$, the objective in \eqref{eq:w-def} is
$(p-L)$-strongly concave, so the regularized best response $w(x,y)$
is well defined and satisfies \eqref{eq:proximal-foc}.

\paragraph{Step 2: The KKT system and the multiplier.}
Because $\phi_{x,y}$ is concave, $c$ is convex, and the Slater condition
holds, the KKT conditions are necessary and sufficient for optimality:
$u\in\mathcal C$ maximizes $\phi_{x,y}$ over $\mathcal C$ if and only if
there exists $\alpha\geq0$ with
$\nabla\phi_{x,y}(u)=\alpha\nabla c(u)$ and $\alpha c(u)=0$. At
$u=u(x,y)$ we have
$\nabla\phi_{x,y}(u)=\nabla_ug(x,u)-\lambda(u-y)=\nabla_ug(x,u)-r$
with $r:=R_x(y)$, so the KKT conditions read
\begin{align}
    -\nabla_u g(x,u)+r+\alpha\nabla c(u)&=0,
    \label{eq:inverseLip-kkt-stationarity}\\
    c(u)&\leq0,
    \qquad
    \alpha\geq0,
    \qquad
    \alpha c(u)=0.
    \label{eq:inverseLip-kkt-complementarity}
\end{align}
Conversely, if $(u,\alpha)$ satisfies
\eqref{eq:inverseLip-kkt-stationarity}--\eqref{eq:inverseLip-kkt-complementarity}
for some vector $r$, then $(u,\alpha)$ is a KKT pair of
\eqref{eq:inverseLip-envelope} with center $y:=u-r/\lambda$, and
therefore $u=u(x,y)$ and $r=R_x(y)$ by \eqref{eq:def-Rx}. This
observation is the basis of the local inversion below.

We next show that $\nabla c(u)\neq0$ whenever $c(u)=0$. Indeed, if
$\nabla c(\bar u)=0$ at some $\bar u$ with $c(\bar u)=0$, then strong
convexity gives
$c(v)\geq c(\bar u)+\frac{\mu_c}{2}\norm{v-\bar u}^2\geq0$ for every
$v$, contradicting the Slater condition. Consequently the multiplier
is unique: $\alpha=0$ if $c(u)<0$, while if $c(u)=0$, taking the inner
product of \eqref{eq:inverseLip-kkt-stationarity} with $\nabla c(u)$
gives
\begin{equation}
    \alpha
    =
    \frac{\langle\nabla_ug(x,u)-r,\nabla c(u)\rangle}{\norm{\nabla c(u)}^2}.
    \label{eq:inverseLip-multiplier-formula}
\end{equation}
We write $\alpha(x,y)$ for this multiplier, so that
$(u(x,y),\alpha(x,y))$ is the unique KKT pair at $(x,y)$.

\paragraph{Step 3: Compactness of the small-residual set and stability of the active set.}
Define
\begin{equation*}
    \mathcal D_{\eta_M}
    :=
    \left\{
        (x,y)\in\K^+:
        \norm{R_x(y)}\leq\eta_M
    \right\}.
\end{equation*}
The set $\K$ is compact by Assumption~\ref{ass:compact-sublevel}, so
its closed unit enlargement $\K^+$ is compact, and
$\mathcal D_{\eta_M}$ is a closed subset of $\K^+$ because $R$ is
continuous. Hence $\mathcal D_{\eta_M}$ is compact.

Fix $\bar z:=(\bar x,\bar y)\in\mathcal D_{\eta_M}$ and write
$\bar u:=u(\bar x,\bar y)$, $\bar\alpha:=\alpha(\bar x,\bar y)$, and
$\bar r:=R_{\bar x}(\bar y)$. We claim that there is an open
neighborhood $N^0$ of $\bar z$ in $\R^{d}\times\R^m$ on which
the active set is constant and $\alpha$ is continuous.

If $c(\bar u)<0$, continuity of $u$ gives an open neighborhood $N^0$
on which $c(u(x,y))<0$, hence $\alpha\equiv0$ on $N^0$.

If $c(\bar u)=0$, then $\bar\alpha>0$ by strict complementarity~(iii),
because $(\bar x,\bar y)\in\K^+$ and $\norm{\bar r}\leq\eta_M$. Suppose,
for contradiction, that there were points $(x_k,y_k)\to(\bar x,\bar y)$
with $c(u_k)<0$, where $u_k:=u(x_k,y_k)$ and $r_k:=R_{x_k}(y_k)$. Then
$\alpha(x_k,y_k)=0$, and \eqref{eq:inverseLip-kkt-stationarity} reduces
to $\nabla_ug(x_k,u_k)=r_k$. By continuity of $u$, $R$, and
$\nabla_ug$, letting $k\to\infty$ gives
$\nabla_ug(\bar x,\bar u)=\bar r$. Comparing with
\eqref{eq:inverseLip-kkt-stationarity} at $(\bar x,\bar y)$ yields
$\bar\alpha\nabla c(\bar u)=0$, which is impossible since
$\bar\alpha>0$ and $\nabla c(\bar u)\neq0$. Hence there is an open
neighborhood $N^0$ of $(\bar x,\bar y)$ on which $c(u(x,y))=0$. On
$N^0$ the multiplier is given by
\eqref{eq:inverseLip-multiplier-formula} with $u=u(x,y)$ and
$r=R_x(y)$, and $\nabla c(u(x,y))\neq0$ there, so $\alpha$ is
continuous on $N^0$.

In both cases, the map $(x,y)\mapsto(u(x,y),\alpha(x,y))$ is
continuous on $N^0$.

\paragraph{Step 4: Local inversion of $R_x$ near $\bar z$.}
We keep $\bar z=(\bar x,\bar y)\in\mathcal D_{\eta_M}$ and the notation
of Step~3, and apply the implicit function theorem to the KKT system in
the variables $(x,r)$. Below, $B(a,\epsilon)$ denotes the open Euclidean
ball of radius $\epsilon$ centered at $a$.

\emph{Inactive case, $c(\bar u)<0$.}
Set $\Phi(x,r,u):=-\nabla_ug(x,u)+r$. Then $\Phi(\bar x,\bar r,\bar u)=0$
and $\partial_u\Phi(\bar x,\bar r,\bar u)=-\nabla^2_{uu}g(\bar x,\bar u)
=H(\bar x,\bar u,0)$, which is nonsingular by~(iv).

\emph{Active case, $c(\bar u)=0$.}
Set
$\Psi(x,r,u,\alpha)
:=\begin{pmatrix}
    -\nabla_ug(x,u)+r+\alpha\nabla c(u) \\ c(u)
\end{pmatrix}$.
Then $\Psi(\bar x,\bar r,\bar u,\bar\alpha)=0$, and the Jacobian of
$\Psi$ with respect to $(u,\alpha)$ at this point is
\begin{equation*}
    M
    =
    \begin{bmatrix}
        H(\bar x,\bar u,\bar\alpha)&\nabla c(\bar u)\\
        \nabla c(\bar u)^\top&0
    \end{bmatrix}.
\end{equation*}
To see that $M$ is nonsingular, let
$M\begin{bmatrix}d\\ \beta\end{bmatrix}=0$, i.e.,
\begin{equation*}
    H(\bar x,\bar u,\bar\alpha)d
    +\beta\nabla c(\bar u)=0,
    \qquad
    \nabla c(\bar u)^\top d=0.
\end{equation*}
The second equation places $d$ in the range of $Z(\bar u)$, so
$d=Z(\bar u)q$ for some $q$. Multiplying the first equation by
$Z(\bar u)^\top$ and using $Z(\bar u)^\top\nabla c(\bar u)=0$ gives
$Z(\bar u)^\top H(\bar x,\bar u,\bar\alpha)Z(\bar u)q=0$. The
reduced-Hessian assumption~(iv) yields $q=0$, hence $d=0$, and then
$\beta\nabla c(\bar u)=0$ with $\nabla c(\bar u)\neq0$ gives
$\beta=0$. Thus $M$ is nonsingular.

\emph{Implicit function chart.}
In both cases, the implicit function theorem provides
$\epsilon_{\bar z}>0$, an open neighborhood $W_{\bar z}$ of
$(\bar u,\bar\alpha)$, and a continuously differentiable map
\begin{equation*}
    (x,r)\mapsto
    \bigl(\hat u_{\bar z}(x,r),\hat\alpha_{\bar z}(x,r)\bigr)\in W_{\bar z}
    \qquad\text{on}\qquad
    V_{\bar z}:=B(\bar x,\epsilon_{\bar z})\times B(\bar r,\epsilon_{\bar z}),
\end{equation*}
with $(\hat u_{\bar z},\hat\alpha_{\bar z})(\bar x,\bar r)
=(\bar u,\bar\alpha)$, such that for every $(x,r)\in V_{\bar z}$ the
pair $(\hat u_{\bar z}(x,r),\hat\alpha_{\bar z}(x,r))$ is the unique
$(u,\alpha)\in W_{\bar z}$ satisfying
\eqref{eq:inverseLip-kkt-stationarity} together with the active-set
condition of $\bar z$, namely $\alpha=0$ in the inactive case and
$c(u)=0$ in the active case. In the active case this is the theorem
applied to $\Psi$. In the inactive case it is the theorem applied to
$\Phi$, which provides a neighborhood $W'$ of $\bar u$; we then set
$\hat\alpha_{\bar z}\equiv0$ and $W_{\bar z}:=W'\times\R$. Shrinking $\epsilon_{\bar z}$, we may assume that
the continuity of
$\partial_r\hat u_{\bar z}$ gives a constant $\kappa_{\bar z}$ with
$\norm{\partial_r\hat u_{\bar z}}\leq\kappa_{\bar z}$ on $V_{\bar z}$.
Each slice $\{x\}\times B(\bar r,\epsilon_{\bar z})$ of the product
$V_{\bar z}$ is convex, so integrating along the segment joining
$(x,r_2)$ and $(x,r_1)$ gives
\begin{equation}
    \norm{\hat u_{\bar z}(x,r_1)-\hat u_{\bar z}(x,r_2)}
    \leq
    \kappa_{\bar z}\norm{r_1-r_2}
    \qquad
    \text{whenever }(x,r_1),(x,r_2)\in V_{\bar z}.
    \label{eq:inverseLip-hatu-lipschitz}
\end{equation}

\emph{Right inverse.}
Define
\begin{equation}
    Y_{\bar z}(x,r)
    :=
    \hat u_{\bar z}(x,r)-\frac{r}{\lambda},
    \qquad (x,r)\in V_{\bar z}.
    \label{eq:inverseLip-local-center}
\end{equation}
Since $(\hat u_{\bar z}(x,r),\hat\alpha_{\bar z}(x,r))$ satisfies
\eqref{eq:inverseLip-kkt-stationarity}--\eqref{eq:inverseLip-kkt-complementarity}
at $(x,r)$, the converse statement in Step~2 gives
$\hat u_{\bar z}(x,r)=u(x,Y_{\bar z}(x,r))$ and
\begin{equation*}
    R_x(Y_{\bar z}(x,r))=r
    \qquad\text{for all }(x,r)\in V_{\bar z},
\end{equation*}
and \eqref{eq:inverseLip-hatu-lipschitz} yields
\begin{equation}
    \norm{Y_{\bar z}(x,r_1)-Y_{\bar z}(x,r_2)}
    \leq
    \Bigl(\kappa_{\bar z}+\frac1\lambda\Bigr)\norm{r_1-r_2}
    \qquad
    \text{whenever }(x,r_1),(x,r_2)\in V_{\bar z}.
    \label{eq:inverseLip-Y-lipschitz}
\end{equation}

\emph{Left inverse on a chart in $(x,y)$-space.}
The map $(x,y)\mapsto(x,R_x(y))$ is continuous and sends $\bar z$ to
$(\bar x,\bar r)\in V_{\bar z}$, and by Step~3 the map
$(x,y)\mapsto(u(x,y),\alpha(x,y))$ is continuous on $N^0$ and sends
$\bar z$ to $(\bar u,\bar\alpha)\in W_{\bar z}$. Hence there is an open
neighborhood $N_{\bar z}\subseteq N^0$ of $\bar z$ such that
\begin{equation*}
    (x,R_x(y))\in V_{\bar z}
    \qquad\text{and}\qquad
    (u(x,y),\alpha(x,y))\in W_{\bar z}
    \qquad
    \text{for every }(x,y)\in N_{\bar z}.
\end{equation*}
Let $(x,y)\in N_{\bar z}$ and $r:=R_x(y)$. By Step~2,
$(u(x,y),\alpha(x,y))$ satisfies
\eqref{eq:inverseLip-kkt-stationarity} at $(x,r)$, and by Step~3 it
satisfies the active-set condition of $\bar z$. The uniqueness property
of the chart gives $u(x,y)=\hat u_{\bar z}(x,r)$, and hence, by
\eqref{eq:def-Rx},
\begin{equation*}
    Y_{\bar z}(x,R_x(y))
    =
    u(x,y)-\frac{1}{\lambda}R_x(y)
    =
    u(x,y)-\bigl(u(x,y)-y\bigr)
    =
    y.
\end{equation*}
Thus $R_x$ and $Y_{\bar z}(x,\cdot)$ are mutually inverse between
$\{y:(x,y)\in N_{\bar z}\}$ and its image. Combining this with
\eqref{eq:inverseLip-Y-lipschitz}, and using that
$(x,R_x(y_i))\in V_{\bar z}$ for $(x,y_i)\in N_{\bar z}$, we obtain
\begin{equation}
    \norm{y_1-y_2}
    =
    \norm{Y_{\bar z}(x,R_x(y_1))-Y_{\bar z}(x,R_x(y_2))}
    \leq
    \Bigl(\kappa_{\bar z}+\frac1\lambda\Bigr)\norm{R_x(y_1)-R_x(y_2)}
    \quad
    \text{whenever }(x,y_1),(x,y_2)\in N_{\bar z}.
    \label{eq:inverseLip-chart-bound}
\end{equation}
The constant here depends on $\bar z$ but not on $x$.

\paragraph{Step 5: Uniform constants.}
The family $\{N_{\bar z}\}_{\bar z\in\mathcal D_{\eta_M}}$ is an open
cover of the compact set $\mathcal D_{\eta_M}$. Extract a finite subcover
$N_1,\dots,N_J$ with associated constants $\kappa_1,\dots,\kappa_J$, and
let $\ell>0$ be a Lebesgue number of the cover
$\{N_j\cap\mathcal D_{\eta_M}\}_{j=1}^J$ of the compact metric space
$\mathcal D_{\eta_M}$: every subset of $\mathcal D_{\eta_M}$ of diameter
less than $\ell$ is contained in a single $N_j$.

Set
\begin{equation}
    \eta_s:=\eta_M,
    \qquad
    \delta_s:=\min\Bigl\{1,\frac{\ell}{2},\frac{\eta_M}{p}\Bigr\},
    \qquad
    \kappa_s:=\max_{1\leq j\leq J}\Bigl(\kappa_j+\frac1\lambda\Bigr).
    \label{eq:inverseLip-final-constants}
\end{equation}

Let $(x,y)\in\K$ and $w=w(x,y)$ satisfy \eqref{eq:inverseLip-region}
with these constants. Then $\norm{w-y}\leq\delta_s\leq1$, so
$(x,w)\in\K^+$, while $(x,y)\in\K\subseteq\K^+$. Moreover,
$\norm{R_x(y)}\leq\eta_s=\eta_M$ by \eqref{eq:inverseLip-region}, and the
first-order condition \eqref{eq:proximal-foc} gives
$\norm{R_x(w)}=\norm{\nabla_yF(x,w)}=p\norm{w-y}\leq p\,\delta_s\leq\eta_M$.
Hence both $(x,y)$ and $(x,w)$ belong to $\mathcal D_{\eta_M}$. They share the same $x$, so
\begin{equation*}
    \norm{(x,y)-(x,w)}=\norm{w-y}\leq\delta_s\leq\frac{\ell}{2}<\ell,
\end{equation*}
and the Lebesgue property places both points in a common chart $N_j$.
Applying \eqref{eq:inverseLip-chart-bound} on $N_j$ yields
\begin{equation*}
    \norm{w-y}
    \leq
    \Bigl(\kappa_j+\frac1\lambda\Bigr)\norm{R_x(w)-R_x(y)}
    \leq
    \kappa_s\norm{\nabla_yF(x,w)-\nabla_yF(x,y)},
\end{equation*}
which is \eqref{eq:inverseLip-bound}. Therefore
Assumption~\ref{ass:inverseLip} holds with the constants
\eqref{eq:inverseLip-final-constants}.
\end{proof}

\begin{remark}[Where each hypothesis is used]
\label{rem:hypothesis-roles}
Hypothesis~(i) makes the inner problem strongly concave, which gives
uniqueness and continuity of $u(x,y)$, Danskin's formula
\eqref{eq:def-Rx}, and the sufficiency of the KKT conditions used in
Steps~2 and~4. Hypothesis~(ii) makes the feasible set compact and
convex, and gives $\nabla c(u)\neq0$ at active points, hence the
uniqueness and the explicit formula
\eqref{eq:inverseLip-multiplier-formula} of the multiplier in Step~2.
Hypothesis~(iii) rules out the degenerate case $\alpha=0$ with
$c(u)=0$, in which the KKT system is not locally a square smooth system
in $(u,\alpha)$ and the active set can change under arbitrarily small
perturbations of $r$; it is what keeps the active set locally constant
in Step~3. Hypothesis~(iv) is the nondegeneracy that makes the implicit
function theorem applicable in Step~4; it is the exact analogue, for the
constrained problem, of \eqref{eq:uniform-nondegeneracy} in
Proposition~\ref{prop:hessian-nondegeneracy}.
\end{remark}

\subsection{Proofs of stationarity transfer theorem}
\label{app:transfer}

\begin{proof}
Since $\rho\geq\rho_c$ and $\G\geq0$,
$    P_{\rho_c}(x,y)
    \leq P_\rho(x,y)
    \leq C_0$.
Hence $(x,y)\in\K$. Also,
\begin{align*}
    \rho\G(x,y)
    =P_\rho(x,y)-F(x,y)
    \nonumber \leq C_0-\underline F_{\K}
    =B.
\end{align*}
Combining above with
\eqref{eq:gap-displacement-lower} yields the coarse localization bound
\begin{displaymath}
    \norm{w-y}
    \leq \sqrt{\frac{2\G(x,y)}{p-L}} \leq
    \frac{D}{\sqrt{\rho}}.
\end{displaymath}
By the first-order condition \eqref{eq:proximal-foc} and
$L$-smoothness give
\begin{align*}
    \norm{\nabla_yF(x,y)}
    &\leq
    \norm{\nabla_yF(x,y)-\nabla_yF(x,w)}
    +\norm{\nabla_yF(x,w)}
    \nonumber\\
    &\leq
    (L+p)\norm{w-y}
    \leq
    \frac{C_y}{\sqrt{\rho}}.
\end{align*}
ALso, $\norm{\nabla_yF(x,w)} =p\norm{w-y} \leq \frac{pD}{\sqrt{\rho}}$. Since $\rho\geq \max\{C_y^2/\eta_s^2, p^2 D^2/\eta_s^2, D^2/\delta_s^2\}$, the displayed bounds give $\norm{\nabla_yF(x,y)}\leq\eta_s$, $\norm{\nabla_yF(x,w)}\leq\eta_s$, and $\norm{w-y}\leq\delta_s$, so \eqref{eq:inverseLip-region} holds. Hence Assumption~\ref{ass:inverseLip} gives
\begin{equation}
    \norm{w-y}\leq\kappa_s\norm{\Delta_y}.
    \label{eq:inverseLip-applied}
\end{equation}

By \eqref{eq:proximal-foc}, $\nabla_yF(x,y)=p(w-y)-\Delta_y$.
Together with \eqref{eq:inverseLip-applied}, this yields
\begin{equation}
    \norm{\nabla_yF(x,y)}
    \leq
    (1+p\kappa_s)\norm{\Delta_y}.
    \label{eq:inverseLip-y-via-delta}
\end{equation}

Applying \eqref{eq:inverseLip-low-level}, we obtain
\begin{align*}
    \varepsilon \geq
    \norm{\nabla_yP_\rho(x,y)}=
    \norm{\nabla_yF(x,y)+\rho\Delta_y} \geq
    \rho\norm{\Delta_y}
    -\norm{\nabla_yF(x,y)} \geq
    (\rho-1-p\kappa_s)\norm{\Delta_y},
\end{align*}
where the equality follows from \eqref{eq:penalty-gradient-y}, the second inequality is by the triangle inequality, and the last inequality is by \eqref{eq:inverseLip-y-via-delta}. Since $\rho\geq2(1+p\kappa_s)$, the denominator is positive and proves the first bound in \eqref{eq:inverseLip-delta-displacement}.  The second bound in
\eqref{eq:inverseLip-delta-displacement} and the first bound in
\eqref{eq:inverseLip-nabla} follow from
\eqref{eq:inverseLip-applied} and
\eqref{eq:inverseLip-y-via-delta}, respectively.

Finally, using \eqref{eq:penalty-gradient-x} first, applying the Lipschitz continuity of $\nabla_xF$ and \eqref{eq:inverseLip-low-level} in the second inequality, and applying second bound of \eqref{eq:inverseLip-delta-displacement} in the third, we have
\begin{align*}
    \norm{\nabla_xF(x,y)}
    &\leq
    \norm{\nabla_xP_\rho(x,y)}
    +\rho\norm{\nabla_xF(x,w)-\nabla_xF(x,y)} \\
    &\leq
    \varepsilon+\rho L\norm{w-y} \leq
    \left(
        1+
        \frac{\rho L\kappa_s}{\rho-1-p\kappa_s}
    \right)\varepsilon.
\end{align*}
This proves the second bound in \eqref{eq:inverseLip-nabla}. Combined with the first, this proves \eqref{eq:inverseLip-game}. Setting $\varepsilon=0$ proves the
last assertion.
\end{proof}

\subsection{Proofs of the complexity theorems}
\label{app:proof-complexity}

\subsubsection{Proof of Theorem~\ref{thm:gd-complexity}}
\label{app:proof-gd}
\begin{proof}
By Lemma~\ref{lem:gap-properties}, $P_\rho$ is $L_\rho$-smooth.
The standard descent lemma applied to
\eqref{eq:penalty-gradient-descent} gives
\eqref{eq:penalty-descent}. Since $P_\rho(z^0)\leq C_0$ by the construction of initialization, \eqref{eq:penalty-descent} gives
\begin{equation*}
    P_\rho(z^k)\leq C_0
    \qquad\text{for all }k.
\end{equation*}
Since $\rho\geq\rho_c$ and $\G\geq0$,
\begin{equation*}
    P_{\rho_c}(z^k)
    \leq P_\rho(z^k)
    \leq C_0 \implies z^k\in\K \qquad\text{for all }k.
\end{equation*}

Summing \eqref{eq:penalty-descent} from $k=0$ to $N-1$ yields
\begin{equation*}
    \frac{1}{2L_\rho}
    \sum_{k=0}^{N-1}
    \norm{\nabla P_\rho(z^k)}^2
    \leq
    P_\rho(z^0)-P_\rho(z^N).
\end{equation*}
Because $P_\rho\geq P_{\rho_c}$, $P_\rho(z^N)\geq P_c^*$. Therefore,
$
    \sum_{k=0}^{N-1}
    \norm{\nabla P_\rho(z^k)}^2
    \leq
    2L_\rho\Delta_c$,
which proves \eqref{eq:penalty-complexity-bound}.

Under \eqref{eq:iteration-complexity}, some iterate satisfies
$\norm{\nabla P_\rho(z^k)} \leq \epsilon/C_\rho$.
This iterate also satisfies $P_\rho(z^k)\leq C_0$, so
Theorem~\ref{thm:stationarity-transfer} gives
\eqref{eq:epsilon-game-stationarity}.
\end{proof}

\subsubsection{Proof of Theorem~\ref{thm:inexact-gd-complexity}}
\label{app:proof-inexact}
\begin{proof}
Since $P_\rho$ is $L_\rho$-smooth, the update
\eqref{eq:inexact-penalty-gradient-descent} gives
\begin{align*}
P_\rho(z^{k+1})
&\leq
P_\rho(z^k)
-\frac{1}{L_\rho}
\langle \nabla P_\rho(z^k),\widetilde g^k\rangle
+\frac{1}{2L_\rho}
\|\widetilde g^k\|^2 \\
&=
P_\rho(z^k)
-\frac{1}{2L_\rho}\|\nabla P_\rho(z^k)\|^2
+\frac{1}{2L_\rho}\|\widetilde g^k-\nabla P_\rho(z^k)\|^2 \\
&\leq
P_\rho(z^k)
-\frac{1}{2L_\rho}\|\nabla P_\rho(z^k)\|^2
+\frac{\rho^2L^2\delta^2}{2L_\rho},
\end{align*}
which proves \eqref{eq:inexact-penalty-descent}.

For the second assertion, observe first that
\eqref{eq:inexact-gradient-error} and
\eqref{eq:inexact-penalty-descent} imply the conditional monotonicity
\begin{equation}
\|\nabla P_\rho(z^j)\|\geq\rho L\delta
\quad\Longrightarrow\quad
P_\rho(z^{j+1})\leq P_\rho(z^j).
\label{eq:conditional-monotonicity}
\end{equation}

Consider the case in which there exists an index $k<N$ such that
\begin{equation}
\|\nabla P_\rho(z^k)\|<\rho L \delta,
\label{eq:first-small-gradient}
\end{equation}
and let $k$ be the first such index. For every $j<k$,
$\|\nabla P_\rho(z^j)\|\geq\rho L \delta$, so
\eqref{eq:conditional-monotonicity} gives
$P_\rho(z^k)\leq P_\rho(z^0)\leq C_0$, and $z^k\in\K$ follows from
$P_{\rho_c}\leq P_\rho$. The bound
\eqref{eq:inexact-penalty-complexity} holds because
$\|\nabla P_\rho(z^k)\|^2<\rho^2L^2\delta^2$.

Otherwise, $\|\nabla P_\rho(z^k)\|\geq\rho L\delta$ for every
$k<N$, and \eqref{eq:conditional-monotonicity} gives
$P_\rho(z^k)\leq P_\rho(z^0)\leq C_0$, and hence $z^k\in\K$, for all
$k\leq N$. Summing
\eqref{eq:inexact-penalty-descent} from $k=0$ to $N-1$ gives
\begin{equation*}
\frac{1}{2L_\rho}
\sum_{k=0}^{N-1}\|\nabla P_\rho(z^k)\|^2
\leq
P_\rho(z^0)-P_\rho(z^N)
+\frac{N\rho^2L^2\delta^2}{2L_\rho}.
\end{equation*}

Since $z^N\in\K$ and $P_{\rho_c}\leq P_\rho$, we have
$P_\rho(z^N)\geq P_{\rho_c}(z^N)\geq P_c^*$, whence
\begin{equation*}
\sum_{j=0}^{N-1}\|\nabla P_\rho(z^j)\|^2
\leq
2L_\rho\Delta_c+N\rho^2L^2\delta^2 .
\end{equation*}
Dividing by $N$, some iterate satisfies
\eqref{eq:inexact-penalty-complexity}. In both cases the selected iterate satisfies $P_\rho(z^k)\leq C_0$ and
\eqref{eq:inexact-penalty-complexity}, so applying Theorem~\ref{thm:stationarity-transfer} gives
\eqref{eq:inexact-game-stationarity}.

Finally, under \eqref{eq:iteration-complexity} and $\delta\leq\epsilon$, some iterate satisfies
\begin{equation}
   \|\nabla_xF(x^k,y^k)\|
    +
    \|\nabla_yF(x^k,y^k)\|
    \leq
    \sqrt{1+C_\rho^2 \rho^2L^2} \epsilon
    \leq (1+C_\rho \rho L)\epsilon.
\end{equation}
\end{proof}


\section{Robust regression example illustrating Assumption~\ref{ass:inverseLip}}
\label{ex:robust-regression}
Consider the one-sample response-adversarial regression problem
\begin{equation}
\min_{x\in\R}\max_{|u|\leq 5}g(x,u),
\qquad
g(x,u):=\frac12(x-u)^2+r(x),
\qquad
r(x):=10x^2+23(\cos x-1).
\end{equation}
The regularizer is smooth and coercive but nonconvex, since
\(r''(x)=20-23\cos x\) and \(r''(0)=-3\). With \(c(u):=u^2-25\) and
\(\lambda=2\), the upper Moreau envelope is
\begin{equation}
\label{eq:regression-envelope}
F(x,y)
:=
\max_{c(u)\leq 0}
\Bigl\{g(x,u)-\frac{\lambda}{2}(u-y)^2\Bigr\}
=
\max_{u^2-25\leq 0}
\Bigl\{\frac12(x-u)^2+r(x)-(u-y)^2\Bigr\}.
\end{equation}
Since \(\nabla_ug(x,u)=u-x\), we have \(L_g=1<\lambda\), so the inner
objective is \(1\)-strongly concave in \(u\) with unique maximizer
\(u(x,y)=\Pi_{[-5,5]}(2y-x)\). Danskin's theorem gives
\(\nabla_xF(x,y)=x-u(x,y)+r'(x)\) and
\begin{equation}
\label{eq:regression-residual}
\nabla_yF(x,y)
=
2\bigl(u(x,y)-y\bigr)
=
\begin{cases}
-2(y+5), & 2y-x\leq-5,\\
2(y-x),  & |2y-x|<5,\\
-2(y-5), & 2y-x\geq5.
\end{cases}
\end{equation}
Because \(u\) is Lipschitz and \(r'(x)=20x-23\sin x\) is
\(43\)-Lipschitz, \(\nabla F\) is globally Lipschitz; moreover, \(\nabla_yF(x,y)\) is
continuous and piecewise affine with slopes \(\pm2\), so
Assumption~\ref{ass:smoothness} holds with \(L_y=2\). We take \(p=4\).

\begin{proposition}
\label{prop:robust-regression-verification}
For the example above with \(p=4\), the following hold.
\begin{enumerate}[label=(\roman*)]
    \item \(F\) is nonconvex in \(x\) and nonconcave in \(y\).
    \item Assumption~\ref{ass:compact-sublevel} holds with \(\rho_c=7\),
    \(C_0= 15\), and \(z_\rho^0=(0,0)\) for every \(\rho\geq7\).
    \item Assumption~\ref{ass:inverseLip} holds.
    \item No global max-player PL inequality, global-gap KL inequality, or
    global error bound to the set of maximizers holds.
\end{enumerate}
Consequently, this example lies within the scope of
Theorems~\ref{thm:stationarity-transfer} and~\ref{thm:gd-complexity}.
\end{proposition}

\begin{proof}
We verify (i)--(iv) in turn.

\begin{enumerate}[label=(\roman*)]
    \item On the region \(|2y-x|<5\), we have \(u(x,y)=2y-x\), and
\eqref{eq:regression-envelope} becomes \(F(x,y)=r(x)+(x-y)^2\). Hence
\(\nabla^2_{xx}F(0,0)=r''(0)+2=-1<0\) and \(\nabla^2_{yy}F(0,0)=2>0\). In
particular, the max-player stationary point \(y=0\) is a strict local
minimum of \(F(0,\cdot)\).

\item
The first bound below follows from taking \(u=0\) in
\eqref{eq:regression-envelope} and \(23(\cos x-1)\geq-46\); the second
from \(|u(x,y)|\leq5\) and \(y^2\leq2(u-y)^2+2u^2\):
\begin{equation*}
F(x,y)\geq\frac{21}{2}x^2-y^2-46,
\qquad
|\nabla_yF(x,y)|^2=4\bigl(u(x,y)-y\bigr)^2\geq2y^2-100.
\end{equation*}

By Lemma~\ref{lem:gap-properties}, \(\G(x,y)\geq|\nabla_yF(x,y)|^2/(2(p+L_y))\). Combining these bounds gives
\begin{align*}
    P_7(x,y)
    &\geq
    F(x,y) + \frac{7}{12}|\nabla_yF(x,y)|^2 \\
    & \geq \frac{21}{2}x^2 + \frac{1}{6}y^2 - \frac{313}{3}.
\end{align*}
Hence, 
\begin{equation*}\label{eq:potential-lower-bound-example}
\K = \{P_7(x,y) \leq  15\} \subseteq \{\frac{21}{2}x^2 + \frac{1}{6}y^2 - \frac{313}{3} \le 15\}  
\end{equation*} 
is compact. At the origin,
\(u(0,0)=0\), \(F(0,0)=0\), and \(\nabla_yF(0,0)=0\), so
\(\G(0,0)=0\) by \eqref{eq:gap-zero-equivalence} and
\(P_\rho(0,0)=0\le 15 = C_0\) for every \(\rho\). Finally,
\eqref{eq:potential-lower-bound-example} gives
\(|x|\leq\sqrt{(313/3+15)\times 2/21}\approx3.371\) on \(\K\); hence, on the unit
enlargement \(\K^+\) of
Proposition~\ref{prop:constrained-envelope-inverseLip},
\begin{equation}
\label{eq:x-bound-enlarged-set}
|x|
\leq
\bar x
:=
1+\sqrt{(313/3+15)\times 2/21}
\approx4.371<5.
\end{equation}

\item We check the hypotheses of Proposition~\ref{prop:constrained-envelope-inverseLip}. Hypothesis (i) holds with $L_g = 1 < \lambda = 2$; (ii) holds with Slater point $u=0$. Let $\bar{x} < 5$ be the bound from \eqref{eq:x-bound-enlarged-set}, and choose $\eta_M =(5-\bar{x})/2>0$. Fix $(x,y) \in \K^+$ satisfying $|R(x)| \le \eta_M$ where $R(x) = \nabla_y F(x,y) = 2(u(x,y)-y)$ and abbreviate $u :=u(x,y)$. The KKT stationarity condition is 
$$-(u-x) +R + 2\alpha u = 0.$$
If $|u|<5$, then $\alpha=0$ and $H(x,u,0) = -\nabla^2_{uu}g(x,u) = -1$ is nonsingular. If $u=5$, then $\alpha = (5-x-R)/10 \ge (5-\bar{x}-\eta_M)/10 = (5-\bar{x})/20>0$. Similarly, if $u=-5$ then $\alpha=(5+x+R)/20 \ge (5-\bar{x}-\eta_M)/10 = (5-\bar{x})/20>0$. Thus, strict complementarity holds at every active KKT point in the required small-residual region, establishing hypothesis (iii). At an active point, $u=\pm5$ and $c'(u) = 2u \neq 0$, so the tangent space is $\{0\}$ and the reduced-Hessian nondegeneracy condition is vacuous. Equivalently, the corresponding KKT matrix has determinant $-4u^2=-100 \neq 0$. Together with the inactive case above, this verifies hypothesis (iv). Proposition~\ref{prop:constrained-envelope-inverseLip} therefore implies Assumption~\ref{ass:inverseLip} holds.

\item The proximal term in \eqref{eq:regression-envelope} is nonpositive and
vanishes at \(y=u\), so \(F^\star(x):=\max_yF(x,y)=\max_{|u|\leq5}g(x,u)\).
At \(x=0\), the value \(F^\star(0)=\frac{25}{2}\) is attained exactly on
\(Y^\star(0)=\{-5,5\}\), whereas \(F(0,0)=0\) and \(\nabla_yF(0,0)=0\).
Hence, for all \(\mu,C,\theta,\kappa>0\), each of
\begin{align}
\frac12|\nabla_yF(x,y)|^2
&\geq
\mu\bigl(F^\star(x)-F(x,y)\bigr),
\nonumber\\
C\bigl(F^\star(x)-F(x,y)\bigr)^\theta
&\leq
|\nabla_yF(x,y)|,
\label{eq:global-gap-KL}\\
\operatorname{dist}\bigl(y,Y^\star(x)\bigr)
&\leq
\kappa|\nabla_yF(x,y)|
\nonumber
\end{align}
fails at \((x,y)=(0,0)\), since the gradient vanishes there while
\(F^\star(0)-F(0,0)=\frac{25}{2}\) and
\(\operatorname{dist}(0,Y^\star(0))=5\). The same obstruction is already
present in the original objective: \(\nabla_ug(0,0)=0\), while
\(\max_{|u|\leq5}g(0,u)-g(0,0)=\frac{25}{2}\).
\end{enumerate}
\end{proof}

\end{document}